\documentclass[11pt]{amsart}

\newcommand{\ds}{\displaystyle}
\newcommand{\cu}{\mathcal{K}}
\newcommand{\ca}{{\mathcal{K}}_0}
\newcommand{\co}{{\mathcal{K}}_{\hbox{reg}}}

\usepackage{color}
\usepackage{enumerate}

\newtheorem{theorem}{Theorem}[section]
\newtheorem{proposition}{Proposition}[section]
\newtheorem{lemma}{Lemma}[section]

\newtheorem{corollary}{Corollary}[section]

\begin{document}

\bibliographystyle{plain}

\begin{center}
{\bf{SOME AFFINE INVARIANTS REVISITED}}\\ {\ } \\
ALINA STANCU \\
Department of Mathematics and Statistics\\ Concordia
University \\ Montreal, Quebec, H3G 1M8, Canada \\
{\em{stancu@mathstat.concordia.ca}}
\end{center}

\

{\bf{Abstract:}} {\em We present several sharp inequalities for the $SL(n)$ invariant $\Omega_{2,n}(K)$ introduced in our earlier work on centro-affine invariants for smooth convex bodies containing the origin. A connection arose with the Paouris-Werner invariant $\Omega_K$ defined for convex bodies $K$ whose centroid is at the origin. We offer two alternative definitions for $\Omega_K$ when $K \in C^2_+$.  The technique employed prompts us to conjecture that any $SL(n)$ invariant of convex bodies with continuous and positive centro-affine curvature function can be obtained as a limit of normalized $p$-affine surface areas of the convex body.}

\

\section{Introduction}

Besides the intrinsic interest in  affine invariants originated in Felix Klein's Erlangen Program, the extension to the Brunn-Minkowski-Firey theory \cite{Lutwak1}, \cite{Lutwak2}, and  very recent connections between affine invariants and fields like stochastic geometry \cite{Barany}, \cite{Gruber} and quantum information theory \cite{Ludwig4}, \cite{LYZ4}, \cite{Paouris}, led to an intense activity in this area of geometric analysis. The renewed interest in affine invariants has benefited also from a systematic approach classifying them, as for example in \cite{Haberl-Ludwig}, \cite{Ludwig3}, \cite{Ludwig1}, \cite{Ludwig2},  and from their use in affine and affine Sobolev inequalities \cite{Haberl-Schuster}, \cite{Haberl-Schuster1}, \cite{Ludwig5}, \cite{LYZ0} - \cite{LYZ2}, \cite{LYZ3}, \cite{Ye-Werner}, \cite{Ye-Werner1} and problems arising in differential geometry \cite{BLYZ}, \cite{Chou}, \cite{Cianchi}, \cite{Haberl-Lutwak}, \cite{Lutwak-Oliker}, \cite{TW-1} - \cite{TW1} which rely on isoperimetric-type functional inequalities. It is the subject of such inequalities that is our primary goal of an on-going project.

 The present paper spun as a follow-up of \cite{IMRN} in which we introduced new $SL(n)$-invariants for smooth convex bodies. We started by searching for sharp affine inequalities satisfied by one such invariant derived, in  a certain sense, from the centro-affine surface area. The resulting inequalities are the subject of the next section. In the process, we encountered a connection to another $SL(n)$ invariant of convex bodies  defined by Paouris and Werner who also related it to quantum information theory \cite{Paouris}. In Section 3, we present two alternative definitions of this invariant. We noted that an additional $SL(n)$ invariant of convex bodies of class $C^2_+$ is defined with analogous techniques. This prompted us to conjecture that $SL(n)$ invariants for convex bodies with continuous and positive centro-affine curvature function can be obtained as limits of normalized $p$-affine surface areas of the convex body.

The setting for this paper is the Euclidean space $\mathbb{R}^n,\ n \geq 2,$ in which we consider convex bodies containing the origin in their interior. Most of the time, we will also require that the convex bodies have smooth boundary with positive Gauss curvature. We will denote the set of such convex bodies by $\co$. However, on several occasions, we will relax the regularity of the boundary to class $C^2$ with positive Gauss curvature and we will use the notation $C^2_+$ to indicate this latter class of convex bodies. The preferred parametrization of a convex body $K$ will be with respect to the unit normal vector, $u \mapsto X_K(u)$, making many functions on the boundary $\partial K$ to be considered as functions on the unit sphere $\mathbb{S}^{n-1}$.

We will denote the Gauss curvature of a convex body  by $\cu$ and its centro-affine curvature by $\ca$. Geometrically, $\ca^{-1/2}$ at a given point of $\partial K$  is, up to a dimension dependent constant, the volume of the centered osculating ellipsoid at that point. Note that the centro-affine curvature is constant if and only if $K$ is a centered ellipsoid. This can also be seen from a lemma due to Petty \cite{Petty} since, analytically, as a function on the unit sphere, the centro-affine curvature is the ratio $\ds \ca (u)= \frac{\cu (u)}{h^{n+1}(u)},\ u \in \mathbb{S}^{n-1}$, where $h$ is the support function of $K$: $h(u)=\max \{{{x}} \cdot u \mid {\bf{x}} \in K \} $ with ${{x}} \cdot u$ denoting the usual inner product in $\mathbb{R}^n$. Two additional notations are deemed necessary. First, $\ds  \mathcal{N}_0 (u):= {\ca}^{-\frac{1}{n+1}} (u) \:
\mathcal{N}(u)$ is the centro-affine normal which is, pointwise, proportional to the (classical) affine normal $\mathcal{N}(u)$, \cite{Leichtweiss}. Finally, we will use $d\mu_K$ to denote the cone measure of $\partial K$ which, given that the Gauss curvature of $K$ is positive, can be expressed by $\ds d\mu_K (x) = h  (\nu (x))\, \frac{1}{\cu} (\nu(x)) \, d\mu_{\mathbb{S}^{n-1}} (\nu (x))$, where $\nu : \partial K \to \mathbb{S}^{n-1}$ is the Gauss map of the boundary of $K$, hence the inverse of the parametrization $X$.

\section{Inequalities for a second order centro-affine invariant}

In \cite{IMRN}, we introduced a class of $SL(n)$ invariants for smooth convex bodies in $\mathbb{R}^n$. For a fixed convex body $K$, these invariants were the first, second, and, for an arbitrary integer $k$, the $k$-th variation of the volume of $K$ while the boundary of the body was subject to a pointwise deformation in the direction of the centro-affine normal by a speed equal to a power of the centro-affine curvature at each specific point. The $p$-affine surface areas introduced by Lutwak \cite{Lutwak2} for $p$ greater than one, later extended to all $p$'s by Meyer-Werner \cite{Meyer-Werner}, are, via this method, part of this class of invariants. To exemplify, and also bring  the reader's attention to a particular such invariant which is one of the main objects of this paper, let us consider the following deformation of a convex body $K$ with smooth boundary:

\begin{equation} \left\{  \begin{alignedat}{2}
\frac{\partial X(u,t)}{\partial t}
&=& {\ca}^{\frac{1}{2}}
(u,t)  \: \mathcal{N}_0(u,t)
\\
X(u,0)&=&\ X_K(u). \hspace{1.5cm}  \end{alignedat}\right.
\label{eq:flows}
\end{equation}

Then, the first variation of $Vol (K)$ is the centro-affine surface area of $K$:
\begin{equation}
\frac{d}{dt}\left(Vol (K) \right)_{t=0}  = -\int_{\partial K} \ca^{\frac{1}{2}} (\nu (x)) \, d\mu_K (x)=- \Omega_n (K)=:\Omega_{1,n}(K),
\end{equation} see \cite{IMRN}.
  Recall that the centro-affine surface area of a convex body is the only one among the $p$-affine surface areas, $\ds \Omega_p (K)=\int_{\partial K} \ca^{\frac{p}{n+p}}\, d\mu_K,$ invariant under $GL(n)$ transformations of the Euclidean space.
Moreover, pursuing an additional variation, we obtain:
\begin{eqnarray}
\Omega_{2,n} (K)&:=&\left( \frac{d^2 \, Vol (K(t)}{dt^2}\right)_{\mid_{t=0}}  \label{eq:O2} \\
&=&\frac{n(n-1)}{2}\, Vol (K^\circ) -
\frac{n-1}{2}\, \int_{\mathbb{S}^{n-1}} h \sqrt{\ca}
\, s (h \sqrt{\ca}, h, \ldots , h)\,
d\mu_{\mathbb{S}^{n-1}},  \nonumber
\end{eqnarray}
where $s(f_1, f_2, \cdots, f_{n-1})$ is an extension of
the mixed curvature function usually defined on $C^2$, here smooth, support functions to
arbitrary smooth functions on the unit sphere $\mathbb{S}^{n-1}$, see \cite{Schneider} page 115 and also \cite{IMRN}.
For the reader familiar with mixed determinants, the following can be taken as definition for the function $s(f_1, f_2, \cdots, f_{n-1})(u):=D (((f_1)_{ij}+\delta_{ij}f_1)(u), ((f_2)_{ij}+\delta_{ij}f_2)(u), \ldots,  ((f_{n-1})_{ij}+\delta_{ij}f_{n-1})(u)), u \in \mathbb{S}^{n-1},$ where $D$ is a mixed determinant and $(\, . \,)_{i}$ represents the covariant differentiation with respect to the $i$-th vector of a positively oriented orthonormal frame on the unit sphere $\mathbb{S}^{n-1}$.

We will show in Proposition \ref{prop:O2} that, in a certain sense, $\Omega_{2,n}(K)$  measures how far $K$ is from being a centered ellipsoid. In preparation, we call the Aleksandrov body, $A_f$, associated with a continuous positive function $f$ on the unit sphere the convex body whose support function $h_f$ is the maximal element of $$\{ h \leq f \mid h:\mathbb{S}^{n-1} \to \mathbb{R} \ {\hbox{support function of a convex body}} \}.$$ If $f$ is itself a support function of a convex body $L$, then $A_f$ is precisely the body $L$. Moreover, in general, $f=h_f$ almost everywhere with respect to the surface area measure of $A_f$. We could not find where this notion first surfaced  in the literature, yet the work \cite{Haberl-Lutwak} gives an excellent background on this notion.
We are now ready to state the following comparison result which we will use in analyzing $\Omega_{2,n}$:

\begin{lemma}[Monotonicity Lemma]
Suppose that $f$ is a strictly positive smooth function on the unit sphere $\mathbb{S}^{n-1}$ and that $h$ is the support function of a convex body $K \subset \mathbb{R}^n$ which belongs to $\co$. Then, denoting by $\ds m:= \min_{\mathbb{S}^{n-1}} \frac{f}{h}$, respectively, $\ds M:= \max_{\mathbb{S}^{n-1}} \frac{f}{h}$, we have
\begin{equation} m \cdot n\, Vol (K) \leq \int_{\mathbb{S}^{n-1}} f s(h, h, ... , h)\, d\mu_{\mathbb{S}^{n-1}} \leq M \cdot n\, Vol (K) \label{eq:mon1}
\end{equation} and, if the Aleksandrov body associated with $f$ has continuous positive curvature function, then  \begin{equation} m^2 \cdot n\, Vol (K) \leq \int_{\mathbb{S}^{n-1}} f s(f, h, ... , h)\, d\mu_{\mathbb{S}^{n-1}} \leq M^2 \cdot n\, Vol (K). \label{eq:mon2}
\end{equation}
\label{lemma:mon}
\end{lemma}
\begin{proof}
Since $K$ belongs to $\co$, $s(h, h, \ldots, h )>0$ on $\mathbb{S}^{n-1}$, thus $m \leq f \leq M$ implies directly (\ref{eq:mon1}).
In fact, we will show that we also have
\begin{equation} m \cdot  V( h, g, h, ... , h) \leq V(f, g, h, \ldots, h) \leq M \cdot  V( h, g, h, ... , h), \label{eq:mon3}
\end{equation}
for any $g$ support function of a convex body, denoted for later use by $K_2$. Indeed, if $f$ itself would be a support function of a convex body, this claim is simply due to the monotonicity of mixed volumes. If $f$ is not a support function, then there exists a large enough constant $c$ so that $f+ch$ is a support function of a convex body, say $L$, with the Gauss parametrization. Moreover, $L \subseteq K_1$, where the latter is the dilation of $K$ by the factor $M+c$. Then, from the monotonicity of mixed volumes, we have that $V(L, K_2, K, \ldots, K) \leq V(K_1, K_2, K, \ldots, K)$. Choosing to represent these mixed volumes through the notation emphasizing the support functions of the two convex bodies, we have $V(f+ch, g, h, \ldots, h) \leq V((M+c)h, g, h, \ldots, h)$. Finally, using the linearity of mixed volumes, we obtain $V(f, g, h, \ldots, h)+ c V(h, g, h, ...., h) \leq (M+c) V(h, g, h, \ldots, h)$ which is, after a trivial simplification, the right inequality of (\ref{eq:mon3}). Similarly, by considering the dilation $K$ of factor $(m+c)$, we obtain a convex body $K_3$ such that $K_3 \subseteq L$ and an  argument analogous with the one above  will imply $m V(h, g,h,  \ldots, h) \leq V(f, g, h, \ldots, h)$.

We will now proceed to prove (\ref{eq:mon2}). Note again that if $f$ would be a support function of a convex body, the claim follows from the monotonicity of mixed volumes.
 If $f$ is not a support function, consider the Aleksandrov body associated to $f$, $A_f$, whose support function we denote by $h_f$. Thus $M\, h \geq f \geq h_f \geq m\, h$ and,  $S_{A_f}$-a.e., $f\circ \nu_{A_f} (x)=h_f (x)$, where $\nu_{A_f}$ is the Gauss map of $\partial A_f$. As, by hypothesis, $A_f$ has a continuous positive curvature function, and by using (\ref{eq:mon3}), we have

\begin{eqnarray}
\int_{\mathbb{S}^{n-1}} f s(f, h, ... , h)\, d\mu_{\mathbb{S}^{n-1}} &=& \int_{\partial A_f} f(\nu_{A_f}^{-1} (x)) s(f, h, \ldots, h) (\nu_{A_f}^{-1}(x))\, dS_{A_f}(x)   \nonumber \\
  &=& \int_{\partial A_f} h_f(\nu_{A_f}^{-1} (x)) s(f, h, \ldots, h) (\nu_{A_f}^{-1}(x))\, dS_{A_f}(x)   \nonumber \\ &=& \int_{\mathbb{S}^{n-1}} h_f s(f, h, ... , h)\, d\mu_{\mathbb{S}^{n-1}} \nonumber \\
&=& n V( f, h_f, h, ... , h)\nonumber \\
& \geq &  m\,V(h, h_f, h, \dots, h) \nonumber \\
& = & m \int_{\mathbb{S}^{n-1}} h_f s(h, h, ... , h)\, d\mu_{\mathbb{S}^{n-1}} \nonumber \\
& \geq & m \int_{\mathbb{S}^{n-1}} m\, h s(h, h, ... , h)\, d\mu_{\mathbb{S}^{n-1}} \nonumber \\
&= & m^2 \cdot n\, Vol (K).
\end{eqnarray}
The second inequality can be proved similarly.
\end{proof}

Consequently, we obtain the following inequalities for $\Omega_{2,n}(K)$.

\begin{proposition} Let $K \in \co$ with the usual notations of $h$ and $\ca$ for the support function, respectively, the centro-affine curvature of $K$ as functions on the sphere $\mathbb{S}^{n-1}$. Then
\begin{enumerate}
\item $\Omega_{n,2}(K) \geq 0$ with equality if and only if $K$ is a centered ellipsoid.
\item If, in addition, the Aleksandrov body associated with $\ds f:=h \sqrt{\ca}$ has continuous positive curvature function,  then $\Omega_{n,2}(K) \leq \ds \frac{(n-1)n}{2}  (M-m) Vol(K)$,
     where $M, m$ are the maximum and minimum of the centro-affine curvature of $K$. The equality occurs if and only if $K$ is a centered ellipsoid.
\end{enumerate} \label{prop:O2}
\end{proposition}

\begin{proof}
\begin{enumerate}
\item The first claim follows immediately from the Minkowski-type inequality we detailed in Lemma 4.3 of \cite{IMRN}
\begin{equation} \left( \int_{\mathbb{S}^{n-1}} f s (f, h,..., h)\, d\mu_{\mathbb{S}^{n-1}} \right)  \left( \int_{\mathbb{S}^{n-1}} h s (h, h,..., h) d\mu_{\mathbb{S}^{n-1}}  \right) \leq \left( \int_{\mathbb{S}^{n-1}} f s (h, h,..., h)\, d\mu_{\mathbb{S}^{n-1}} \right)^2, \nonumber
\end{equation} where $f$ is an arbitrary smooth function on the sphere, while $h$ is a smooth support function of a convex body.
It suffices to apply this inequality to the second term of $\Omega_{n,2}(K)$ with $\ds f:=h \sqrt{\ca}$ to obtain
\begin{equation}
\Omega_{2,n} (K) \geq \frac{n(n-1)}{2}\, Vol (K^\circ) -
\frac{n-1}{2n}\, \frac{\Omega_n^2 (K)}{Vol (K)}  \nonumber
\end{equation}
from which the result follows by H\"older's inequality
\begin{equation}
Vol (K^\circ) \cdot Vol (K)= \frac{1}{n^2} \left(\int_{\partial K} \ca \, d\mu_K \right) \cdot  \left(\int_{\partial K} d\mu_K \right)  \geq \frac{1}{n^2}\,   \left(\int_{\partial K} \sqrt{\ca}\, d\mu_K \right)^2.
\end{equation} Note that the equality is attained if and only if $\ca$ is constant on $\mathbb{S}^{n-1}$, hence  if and only if $K$ is a centered ellipsoid.
\item By taking $\ds f=h \sqrt{\ca}$ with $m \leq \ca \leq M$, we can apply (\ref{eq:mon2}), \begin{eqnarray}
\Omega_{2,n} (K)
&=&\frac{n(n-1)}{2}\, Vol (K^\circ) -
\frac{n-1}{2}\, \int_{\mathbb{S}^{n-1}} h \sqrt{\ca}
\, s (h \sqrt{\ca}, h, \ldots , h)\,
d\mu_{\mathbb{S}^{n-1}},  \nonumber  \\ &\leq &  \frac{n(n-1)}{2}\, \frac{1}{n} \int_{\partial K} \ca \, d\mu_K -\frac{n(n-1)}{2}\, m \, Vol (K)  \nonumber  \\ &\leq &  \frac{n(n-1)}{2}\, \left(M -m \right) \, Vol (K).
\end{eqnarray}
Equality is attained if and only if $M=m$ which implies, as before, that  $K$ is a centered ellipsoid. Note that we have only used the left-hand side inequality of (\ref{eq:mon2}). It so happens that the right-hand side inequality of (\ref{eq:mon2}) follows for this choice of function $f$ from the positivity of $\Omega_{2,n} (K)$ for any $K \in \co$.
\end{enumerate}
\end{proof}

Further, the previous result implies additional isoperimetric-type inequalities.

\begin{theorem} If $K\in \co$,  the following $Gl(n)$-invariant inequality holds
$$\frac{1}{n^2}\, \Omega_n^2(K) \leq Vol (K) \cdot Vol (K^{\circ}) \leq  \frac{2}{n(n-1)} \min \{ Vol (K) \cdot \Omega_{n,2} (K),  Vol (K^\circ) \cdot \Omega_{n,2} (K^\circ)\}  + \frac{1}{n^2}\, \Omega_n^2(K),$$  and equality occurs if and only if $K$ is a centered ellipsoid.

If, in addition, $K$ is such that the Aleksandrov body associated with $\ds f:=h \sqrt{\ca}$ has continuous positive curvature function and $\ds \frac{M}{m} \leq \frac{1+\sqrt{5}}{2}$, the golden ratio, then the following $Gl(n)$-invariant inequality holds:
$$\frac{1}{n^2}\, \Omega_n^2(K) \leq Vol (K) \cdot Vol (K^{\circ}) \leq \frac{1}{n^2}\, \Omega_n^2(K) \, \left[1- \frac{M-m}{\sqrt{Mm}} \right]^{-1}$$ with equality if and only if $K$ is a centered ellipsoid. \label{theorem:vol}
\end{theorem}
\begin{proof}
The left-hand inequality follows immediately from H\"older's inequality.  In fact, this easy remark motivated a search for
an upper bound of the volume product $Vol (K) \cdot Vol (K^{\circ})$ in terms of the centro-affine surface area or, in other words, a reverse isoperimetric-type inequality.

Toward this goal, note that
the sign of $\Omega_{2,n} (K)$ translates into the following $Gl(n)$-invariant inequality:
$$\frac{1}{n^2}\, \Omega_n^2(K) \leq Vol (K) \cdot Vol (K^{\circ}) \leq \frac{2}{n(n-1)} Vol (K) \cdot \Omega_{n,2} (K) + \frac{1}{n^2}\, \Omega_n^2(K),$$  with equality if and only if $K$ is a centered ellipsoid. Apply the same inequality with the roles of $K$ and $K^\circ$ reversed and use the fact that $\Omega_n(K)=\Omega_n (K^\circ)$, \cite{Hug}, \cite{Ludwig2}, \cite{Ye-Werner}. Therefore,
$$\frac{1}{n^2}\, \Omega_n^2(K) \leq Vol (K) \cdot Vol (K^{\circ}) \leq \frac{2}{n(n-1)} \min \{ Vol (K) \cdot \Omega_{n,2} (K),  Vol (K^\circ) \cdot \Omega_{n,2} (K^\circ)\} + \frac{1}{n^2}\, \Omega_n^2(K),$$  with equality if and only if $K$ is a centered ellipsoid.

From Proposition \ref{prop:O2}, $$ \frac{2}{n(n-1)} \, Vol (K) \cdot \Omega_{n,2} (K) \leq (M-m)\, Vol^2 (K)$$ and $$\frac{2}{n(n-1)} \, Vol (K^\circ) \cdot \Omega_{n,2} (K^\circ) \leq (M^\circ-m^\circ)\, Vol^2 (K^\circ),$$ thus $$\frac{2}{n(n-1)} \min \{ Vol (K) \cdot \Omega_{n,2} (K),  Vol (K^\circ) \cdot \Omega_{n,2} (K^\circ)\} \leq \sqrt{(M-m)(M^\circ - m^\circ)} \, Vol (K) \cdot Vol (K^\circ).$$ Here $m^\circ$ and $M^\circ$ are the minimum, respectively, the maximum of the centro-affine curvature of $\partial K^\circ$.

For any point of $\partial K$, $x$, there exists a point $y$ on $\partial K^\circ$ such that $\ca (x) \cdot \ca^\circ (y)=1$, see \cite{Hug}, thus $\displaystyle M \cdot m^\circ =1$ and $m \cdot M^{\circ}=1$ otherwise a contradiction with one of the definitions of $m^\circ,\ M^\circ$ occurs. Hence
$$\sqrt{(M-m)(M^\circ - m^\circ)} = \sqrt{ (M-m) \left(\frac{1}{m} - \frac{1}{M} \right)} = \frac{M-m}{\sqrt{Mm}},$$ which is less or equal to $1$ if and only if $M/m$ is less or equal to the golden ratio above.

Thus $$Vol (K) \cdot Vol (K^\circ) \leq \frac{M-m}{\sqrt{Mm}} \cdot Vol (K) \cdot Vol (K^\circ) + \frac{1}{n^2}\, \Omega_n^2(K)$$ which implies the right-hand side inequality. The equalities follow as before from $M=m$ equivalent to constant centro-affine curvature along the boundary $\partial K$.
\end{proof}

Note that in the next proposition we drop the smoothness assumption on the boundary of $K$ to class $C^2$.

\begin{proposition} For any $p >1$, and any $K \in C^2_+$ with the origin in its interior, we have
\begin{equation}\frac{\Omega_p^{n+p}(K)}{Vol^{n-p} (K)} \leq n^{p-1}\, \left(Vol (K) \cdot Vol (K^{\circ})\right)^{p-1} \cdot \frac{\Omega^{n+1}(K)}{Vol^{n-1} (K)}. \label{eq:affs} \end{equation} The equality holds if and only if $p=1$ or  $K$ is a centered ellipsoid.

The opposite inequality holds for $p<1$, $p \neq -n$. \label{prop:two}
\end{proposition}
\begin{proof}
Note that, for any $p \neq -n$, \begin{equation} \Omega_p(K)=\int_{\partial K} \ca^{\frac{p}{n+p}} \, d\mu_K = \int_{\partial K} \left(\ca^{\frac{n}{n+1}}\right)^{\frac{p-1}{n+p}} \, d\sigma_K,
\end{equation}
where $d\sigma_K $ is the affine surface area measure, in other words the Blaschke metric, of $K$. As the function $\ds p \mapsto \frac{p-1}{n+p}$ is concave for $p \geq 1$ and convex for $p \leq 1$, we apply the appropriate Jensen's inequality for each range and the normalized measure $\ds \frac{1}{\Omega (K)}\, d\sigma_K $. If $p \geq 1$, we obtain \begin{equation} \left( \frac{n Vol (K^\circ)}{\Omega (K) }\right)^{\frac{p-1}{n+p}} \geq \frac{\Omega_p(K)}{\Omega(K)} \ \ \ \Leftrightarrow \ \ \ \Omega_p (K) \leq \left(n \, Vol (K^\circ) \right)^{\frac{p-1}{n+p}} \cdot \Omega^{\frac{n+1}{n+p}} (K) \end{equation} with equality if and only if $p=1$ or $K$ is a centered ellipsoid. A re-arrangement of terms, gives (\ref{eq:affs}). The proof of the reverse inequality in the case $p \leq 1$ is perfectly similar.
\end{proof}

\begin{corollary} For any convex body $K \in \co$,
$$\ds n^n \left[\frac{2}{n-1} Vol (K) \cdot \Omega_{n,2} (K) + \frac{1}{n}\, \Omega_n^2(K) \right] \geq \frac{\Omega^{n+1}(K)}{Vol^{n-1} (K)} \geq \frac{\Omega_n^{2n}(K)}{{[(2/(n-1)) \, \Omega_{n,2}(K) Vol (K) +\Omega_n^{2}(K)/n]}^{n-1}},$$ with equality iff $K$ is a centered ellipsoid. \label{cor:co}
\end{corollary}
\begin{proof}
Apply the previous result for $p=0$ and, respectively, $p=n$, and use the bounds on $Vol (K) \cdot Vol(K^\circ)$ from Theorem \ref{theorem:vol}.
\end{proof}

\begin{corollary}[Isoperimetric-like Inequality] For any $K \in C^2_+$ with the centroid at the origin, and any $T \in Sl(n)$,  \begin{equation}\frac{S^n(TK)}{Vol^{n-1} (K)} \geq \frac{\omega_n^{2n-3}}{n}\, \max \left\{\frac{\Omega_n^{n+1} (K) }{\Omega^{n+1} (K)/Vol^{n-1} (K)}, \left(\frac{\Omega^{n+1} (K)}{Vol^{n-1}(K)}\right)^{n-1} \right\}, \end{equation} where $S (TK)$ stands for the surface area of $TK$ and $\omega_n$ is the volume of the unit ball $x_1^2 + \ldots + x_n^2=1$ in $\mathbb{R}^n$.  Equality occurs if and only if $K$ is a centered ellipsoid and $T$ is the linear transformation of determinant one such that $TK$ is a ball.

Hence
\end{corollary}
\begin{proof}
Consider $p=n$ in the inequality of Proposition \ref{prop:two} to obtain
\begin{equation} \Omega_n^{2n}(K) \leq n^{n-1} [Vol (K) \cdot Vol (K^\circ)]^{n-1} \cdot \frac{\Omega^{n+1}(K)}{Vol^{n-1}(K)}.
\end{equation} From the classical isoperimetric inequality, $\ds Vol^{n-1}(K) \leq \left(Vol^{n-1}(B)/S^n (B) \right) S^n(K)$, where $B$ is the unit ball as above. On the other hand, by Blaschke-Santal\'o inequality, $Vol(K) \cdot Vol (K^\circ) \leq (Vol (B))^2$.

Therefore
\begin{equation} \Omega_n^{2n}(K) \leq n^{n-1}\frac{Vol^{3(n-1)}(B)}{S^n (B)}\,  \frac{S^n (K)}{Vol^{n-1}(K)} \cdot \frac{\Omega^{n+1}(K)}{Vol^{n-1}(K)},
\end{equation} where all quantities, except $S(K)$, are invariant under linear transformations of determinant one. Hence, the conclusion follows as $\ds n^{n-1}\,\frac{Vol^{3(n-1)}(B)}{S^n (B)}=\frac{\omega_n^{2n-3}}{n}.$ To analyze the equality case one needs to take $T$ to be the linear transformation of determinant one  minimizing the surface area of $K$ and note that all other equalities hold if and only if $K$ is a centered ellipsoid.

We will now use $p=0$ in Proposition \ref{prop:two}, to obtain
\begin{eqnarray}
\frac{\Omega^{n+1} (K)}{Vol^{n-1}(K)} &\leq & n \, Vol (K) \cdot Vol (K^\circ) \leq  n \, \frac{V(B)}{S(B)^{n/(n-1)}} \cdot S(K)^{n/(n-1)} \cdot Vol (K^\circ) \nonumber \\ &\leq & n \, \frac{V(B)^3}{S(B)^{n/(n-1)}} \cdot \frac{S(K)^{n/(n-1)}}{Vol (K)} = n^{1-\frac{n}{n-1}}\, \omega_n^{3-\frac{n}{n-1}}\,  \frac{S(K)^{n/(n-1)}}{Vol (K)},
\end{eqnarray} relying again on Blaschke-Santal\'o inequality.


From here, \begin{equation}\frac{S^n(TK)}{Vol^{n-1} (K)} \geq \frac{\omega_n^{2n-3}}{n}\,  \left(\frac{\Omega^{n+1} (K)}{Vol^{n-1}(K)}\right)^{n-1}, \end{equation} with the same condition for the equality case as above.
\end{proof}

One can use K. Ball's reverse isoperimetric ratio which gives an upper bound on $\ds \frac{S^n(TK)}{Vol^{n-1} (K)}$ by the corresponding ratio for the regular solid simplex in $\mathbb{R}^n$ (or the solid cube in the centrally-symmetric case), \cite{Ball1}, \cite{Ball2}, in the above corollary to get lower bounds on the affine isoperimetric ratio of bodies in $C^2_+$. However, these bounds will not be sharp.

As in Corollary \ref{cor:co}, one can drop  the requirement that the centroid of $K$ is at the origin, consider $K \in \co$, and use the upper bound on the volume product from Theorem \ref{theorem:vol} instead of Blaschke-Santal\'o inequality, to obtain $SL(n)$ invariant lower bounds on the isoperimetric ratio $\ds S(TK)^n/Vol (K)^{n-1}$.

Finally, we include the next corollary, due to \cite{Paouris}, which follows immediately from Proposition \ref{prop:two}.
\begin{corollary} For any convex body $K$ of class $C^2_+$ containing the origin in its interior,
\begin{equation}\Omega_K \leq  \frac{\Omega^{n+1}(K)}{(nVol (K^\circ))^{n+1},}
\end{equation} where $\ds \Omega_K:=\lim_{p \to \infty} \left( \frac{\Omega_p (K)}{n Vol (K^\circ)}\right)^{n+p} $ is the affine invariant introduced by Paouris and Werner in \cite{Paouris}. The equality occurs if and only if $K$ is a centered ellipsoid.
\end{corollary}
Note that in \cite{Paouris}, for certain considerations, the invariant $\Omega_K$ has been defined for convex bodies whose centroid is at the origin, yet the above definition makes sense for any convex body $K$ of class $C^2_+$ containing the origin in its interior for which one can show as in \cite{Paouris} that the limit exists.

\section{More on the Paouris-Werner invariant}

 Motivated by the earlier occurrence of $\Omega_K$, we would like to give here a couple of other  definitions of this invariant when $K$ belongs to $C^2_+$.
To do so, let us also recall that Paouris and Werner showed in \cite{Paouris} that $\Omega_K$ is related to the Kullback-Leibler
divergence $D_{KL}$ of two specific probability measures $P$, $Q$ on $\partial K$ via the relation $\ds D_{KL}(P \| Q)= \ln \left( \frac{Vol (K)}{Vol (K^\circ)} \Omega_K^{-1/n} \right),$ where, in slightly different terms than in \cite{Paouris}, 
$$\ds D_{KL}(P \| Q):=\frac{1}{n Vol (K^\circ)} \int_{\partial K} \ca \ln \left( \ca \, \frac{Vol (K)}{Vol (K^\circ)} \right) \, d\mu_K.$$ Hence, it is useful to note the identity \begin{equation} \ln (\Omega_K)=-\frac{1}{Vol (K^\circ)} \int_{\partial K} \ca \ln \ca \, d\mu_K, \label{eq:OK}  \end{equation} and note that, in this paper, we assume only that the origin is contained in the interior of the convex body $K$.

\begin{proposition}
For any $K$ of class $C^2_+$ containing the origin in its interior, and any integer $p>1$, the following $Gl(n)$-invariant inequalities hold
\begin{equation}
\Omega_n^2(K)  \geq \frac{\left(\Omega_{n/3}(K) \right)^{4}}{(n Vol(K))^{2}}\geq \frac{\left(\Omega_{n/7}(K) \right)^{8}}{(n Vol(K))^{6}}\geq ... \geq \frac{\left(\Omega_{n/(2^p-1)}(K) \right)^{2^p}}{(n Vol(K))^{2^p-2}} \geq \ldots , \label{eq:seq1}
\end{equation}
or, alternately,
\begin{equation}
\Omega_n^2(K)\geq \frac{\left(\Omega_{3n}(K^\circ) \right)^{4}}{(n Vol(K))^{2}} \geq \frac{\left(\Omega_{7n}(K^\circ) \right)^{8}}{(n Vol(K))^{6}} \geq ... \geq \frac{\left(\Omega_{n(2^p-1)}(K^\circ) \right)^{2^p}}{(n Vol(K))^{2^p-2}} \geq \ldots , \label{eq:seq2}
\end{equation}
\begin{equation}
\Omega_n^2(K)\geq \frac{\left(\Omega_{3n}(K) \right)^{4}}{(n Vol(K^\circ))^{2}} \geq \frac{\left(\Omega_{7n}(K) \right)^{8}}{(n Vol(K^\circ))^{6}} \geq ... \geq \frac{\left(\Omega_{n(2^p-1)}(K) \right)^{2^p}}{(n Vol(K^\circ))^{2^p-2}} \geq \ldots . \label{eq:seq3}
\end{equation} In all sequences, all equalities hold if and only if $K$ is a centered ellipsoid (which is the only reason why we did not include $p=1$ in the statement). \label{prop:decr}
\end{proposition}
\begin{proof}
Note that (\ref{eq:seq1}) and (\ref{eq:seq2}) are equivalent through the equality $\Omega_q (K)= \Omega_{n^2/q}(K^\circ)$, \cite{Hug}, \cite{Ludwig2}, \cite{Ye-Werner}. Same goes for (\ref{eq:seq3}) due to $\Omega_n(K)=\Omega_n(K^\circ)$ and interchanging the roles of $K$ and $K^\circ$ in the previous sequence of inequalities. Thus, it suffices to prove (\ref{eq:seq1}).

We will use the concavity of the function $x \mapsto \sqrt{x}$ on $(0, \infty)$ and Jensen's inequality as follows:
\begin{equation}
\left(\frac{\Omega_n (K)}{n Vol (K)}\right)^{1/2} = \left(\int_{\partial K} \sqrt{\ca}\, \frac{d\mu_K}{n\, Vol (K)}\right)^{1/2} \geq \int_{\partial K} \sqrt[4]{\ca}\, \frac{1}{n\, Vol (K)}\, d\mu_K,
\end{equation}
thus $$\left(\frac{\Omega_n (K)}{n Vol (K)}\right)^{1/2}  \geq \frac{\Omega_{n/3} (K)}{n Vol (K)},$$ which is, after raising both sides to power four, the first inequality of (\ref{eq:seq1}). Re-iterate now the same argument for $\Omega_{n/3}(K)$:
\begin{equation}
\left(\frac{\Omega_{n/3} (K)}{n Vol (K)}\right)^{1/2} = \left(\int_{\partial K} \sqrt[4]{\ca}\, \frac{d\mu_K}{n\, Vol (K)}\right)^{1/2} \geq \int_{\partial K} \sqrt[8]{\ca}\, \frac{1}{n\, Vol (K)}\, d\mu_K,
\end{equation} which translates into  $$\left(\frac{\Omega_{n/3} (K)}{n Vol (K)}\right)^{1/2}  \geq \frac{\Omega_{n/7} (K)}{n Vol (K)}.$$ Hence
$$\Omega_n(K) \geq \frac{\Omega_{n/3}^2(K)}{n Vol (K)} \geq \frac{\Omega_{n/7}^4(K)}{(n Vol (K))^3}$$
and so on, the sequence is obtained by iterating the argument.
\end{proof}

\begin{theorem}[Alternative Definition of $\Omega_K$] For any $K$ of class $C^2_+$ containing the origin in its interior, the scaling invariant sequence
$\ds \left\{ \frac{\left(\Omega_{n(2^p-1)}(K) \right)^{2^p}}{(n Vol(K^\circ))^{2^p}} \right\}_{p \in \mathbb{N},\ p \geq 1}$ converges and \begin{equation} \lim_{p \to \infty} \left(\frac{\Omega_{n(2^p-1)}(K) }{n Vol(K^\circ)}\right)^{2^p} = \Omega_K.
\end{equation}
\end{theorem}
\begin{proof}
By (\ref{eq:seq3}), the positive sequence $\ds \frac{\left(\Omega_{n(2^p-1)}(K) \right)^{2^p}}{(n Vol(K^\circ))^{2^p-2}}$ is decreasing, thus converges. Therefore, so does the sequence above whose general term differs from general term of the former sequence by a factor of $\ds (nVol (K^\circ))^{-2}$.

Let $\ds q := n({2^p-1}),$ and, similarly with  Proposition 3.6 in \cite{Paouris}, consider

\begin{eqnarray}
\ln \left[\lim_{p \to \infty}\frac{\left(\Omega_{n(2^p-1)}(K) \right)^{2^p}}{(n Vol(K^\circ))^{2^p}}\right] &=& \lim_{p \to \infty} 2^p \ln \left(\frac{\Omega_{n(2^p-1)}(K)}{n Vol(K^\circ)} \right) \nonumber \\ &=& - \frac{2^p}{\ln 2} \, \frac{\frac{d}{dp} \left(\Omega_{n(2^p-1)} (K) \right)}{\Omega_{n(2^p-1)} (K)} \nonumber \\ &=& -\lim_{p \to \infty} \frac{2^p}{\ln 2} \, \frac{\frac{d}{dq} \left(\int_{\partial K} \exp \left( \ln \ca^{\frac{q}{n+q}}\right)\, d\mu_K\right) \frac{dq}{dp}}{\Omega_{n(2^p-1)} (K)} \nonumber \\ &=&-\lim_{p \to \infty}   {2^{2p}} \, \frac{\frac{d}{dq} \left(\int_{\partial K} \exp \left( \ln \ca^{\frac{q}{n+q}}\right)\, d\mu_K\right) }{\Omega_{n(2^p-1)} (K)} \nonumber \\ &=&-\lim_{p \to \infty}   {2^{2p}} \, \frac{\left(\int_{\partial K} \exp \left( \ln \ca^{\frac{q}{n+q}}\right)\, \ln  ( \ca)\, \frac{n}{(n+q)^2} \, d\mu_K\right) }{\Omega_{n(2^p-1)} (K)} \nonumber \\ &=&-n
\lim_{p \to \infty}\frac{\int_{\partial K}  \ca^{\frac{2^p-1}{2^p}}\, \ln  ( \ca) \, d\mu_K }{\Omega_{n(2^p-1)} (K)} \nonumber \\
&=&- n\,
\frac{\int_{\partial K}  \ca\, \ln  ( \ca) \, d\mu_K }{n \, Vol (K^\circ)}= \ln (\Omega_K). \nonumber
\end{eqnarray} The last equality, due to (\ref{eq:OK}), completes the proof.
\end{proof}
Following from the monotonicity of the sequence (\ref{eq:seq3}), we have
\begin{corollary}
For any $K$ of class $C^2_+$ containing the origin in its interior, and any integer $p \geq 1$, \begin{equation}{\Omega_K}\cdot{(n Vol (K^\circ))^2} \leq \frac{(\Omega_{n(2^p-1)}(K))^{2^p} }{(n Vol(K^\circ))^{2^p-2}},\end{equation} in particular $\ds {\Omega_K}\cdot{(n Vol (K^\circ))^2} \leq \Omega_n^2(K)$,  with equalities everywhere if and only if $K$ is a centered ellipsoid.
\end{corollary}

\begin{corollary}
For any $K$ of class $C^2_+$ containing the origin in its interior, and any integer $p \geq 1$, \begin{equation} {\Omega_K}\cdot \Omega_{K^\circ} \leq \frac{(\Omega_{n(2^p-1)}(K) \cdot \Omega_{n(2^p-1)}(K^\circ))^{2^p} }{(n^2 Vol (K) \cdot Vol(K^\circ))^{2^p}},\end{equation} in particular $\ds {\Omega_K}\cdot \Omega_{K^\circ} \leq \frac{\Omega_n^2(K) \cdot\Omega_n^2(K^\circ)}{(n^2 Vol (K) \cdot Vol(K^\circ))^{2}}$,  with equalities everywhere if and only if $K$ is a centered ellipsoid in which case the right-hand sides of the two inequalities are equal to $1$.
\end{corollary}

The definition of $\Omega_K$ can be extended to affine surface areas of negative exponent using a similar result with Proposition \ref{prop:decr}:

\begin{theorem}[Second Alternative Definition of $\Omega_K$] For any $K$ of class $C^2_+$ containing the origin in its interior, the sequence
$\ds \left\{ \left(\frac{\Omega_{-(n+2^p)}(K^\circ) }{n Vol(K)} \right)^{2^p} \right\}_{p \in \mathbb{N},\ p \geq 1}$ converges and \begin{equation} \lim_{p \to \infty} \left(\frac{\Omega_{-(n+2^p)}(K^\circ) }{n Vol(K)} \right)^{2^p} = \Omega_K^{-1}.
\end{equation}
\end{theorem}

\begin{proof}
By applying again Jensen's inequality for the concave function $x  \mapsto \sqrt{x}$, $x >0$, we have, for any integer $p \geq 1$,
\begin{equation}
\int_{\partial K} \ca^{-\frac{n}{2}}\, d\mu_K \geq \frac{\int_{\partial K} \ca^{-\frac{n}{4}}\, d\mu_K}{n \, Vol (K)} \geq \frac{\int_{\partial K} \ca^{-\frac{n}{8}}\, d\mu_K}{(n \, Vol (K))^3} \geq \ldots \geq \frac{\int_{\partial K} \ca^{-\frac{n}{2^p}}\, d\mu_K}{(n \, Vol (K))^{2^p-1}} \geq \ldots ,
\end{equation}
therefore the sequence of general term 

$\ds   \left(\frac{\Omega_{-(n+2^p)}(K^\circ) }{n Vol(K)} \right)^{2^p}= \left(\frac{\Omega_{-n^2/(n+2^p)}(K) }{n Vol(K)} \right)^{2^p} = \left(\frac{1}{n\, Vol (K)} \cdot \frac{(\Omega_{-n^2/(n+2^p)}(K))^{2^p} }{(n Vol(K))^{2^p-1} } \right)$ 

\noindent is monotone. Interchanging $K$ with $K^\circ$, we conclude that the sequence $\ds \left\{  \left(\frac{\Omega_{-(n+2^p)}(K) }{n Vol(K^\circ)} \right)^{2^p} \right\}_{p \in \mathbb{N}, \ p \geq 1}$
is monotone.

We now proceed as in the previous theorem with
\begin{eqnarray}
\ln \left[\lim_{p \to \infty}\frac{\left(\Omega_{-(n+2^p)}(K) \right)^{2^p}}{(n Vol(K^\circ))^{2^p}}\right] &=& \lim_{p \to \infty} 2^p \ln \left(\frac{\Omega_{-(n+2^p)}(K)}{n Vol(K^\circ)} \right) \nonumber \\ &=& - \frac{2^p}{\ln 2} \, \frac{\frac{d}{dp} \left(\Omega_{-(n+2^p)} (K) \right)}{\Omega_{-(n+2^p)} (K)} \nonumber \\ &=& -\lim_{p \to \infty} \frac{2^p}{\ln 2} \, \frac{\frac{d}{dp} \left(\int_{\partial K} \exp \left( \ln \ca^{\frac{n}{2^p}+1}\right)\, d\mu_K\right)}{\Omega_{-(n+2^p)} (K)} \nonumber \\  &=&n\, \lim_{p \to \infty}   \frac{\left(\int_{\partial K} \exp \left( \ln \ca^{\frac{n+2^p}{2^p}}\right)\, \ln  ( \ca) \, d\mu_K\right) }{\Omega_{-(n+2^p)} (K)} \nonumber \\
&=& n\,
\frac{\int_{\partial K}  \ca\, \ln  ( \ca) \, d\mu_K }{n \, Vol (K^\circ)}=- \ln (\Omega_K), \nonumber
\end{eqnarray} and, using (\ref{eq:OK}), we complete the proof of the theorem.
\end{proof}
While it is known that integrals of the form $\ds \int_{\partial K} \phi (\ca)\, d\mu_K$ are $SL(n)$-invariant, see also \cite{Ludwig1}, \cite{Ludwig2}, considering the results in \cite{Paouris}, and others, including for example the next theorem, we conjecture that the set of $p$-affine surface areas, with algebraic operations, can generate, by taking the {\em closure}, all integrals of the above form.

\begin{theorem}
For any $K$ of class $C^2_+$ containing the origin in its interior, the $SL(n)$-invariant $\ds \Lambda (K) := \exp \left[\frac{1}{n Vol (K)}\, \int_{\partial K} \ln (\ca) \, d\mu_K \right]$ is the limit, as $p \to + \infty$, of the sequence $\ds \left\{ \left( \frac{\Omega_{-\frac{n}{2^p}}(K) }{n\, Vol (K)} \right)^{2^p} \right\}_{p \in \mathbb{N},\ p > 1}.$
\end{theorem}

\begin{proof}
The claim follows directly from \begin{eqnarray}
\ln \left[\lim_{p \to \infty}\frac{\left(\Omega_{-n/2^p}(K) \right)^{2^p}}{(n Vol(K^\circ))^{2^p}}\right] &=& \lim_{p \to \infty} 2^p \ln \left(\frac{\Omega_{-n/2^p}(K)}{n Vol(K^\circ)} \right) \nonumber \\ &=& - \frac{2^p}{\ln 2} \, \frac{\frac{d}{dp} \left(\Omega_{-n/2^p} (K) \right)}{\Omega_{-n/2^p} (K)} \nonumber \\ &=& -\lim_{p \to \infty} \frac{2^p}{\ln 2} \, \frac{\frac{d}{dp} \left(\int_{\partial K} \exp \left( \ln \ca^{-\frac{1}{2^p-1}}\right)\, d\mu_K\right)}{\Omega_{-n/2^p} (K)} \nonumber \\  &=& \lim_{p \to \infty} \frac{2^{2p}}{(2^p-1)^2} \,   \frac{\left(\int_{\partial K} \exp \left( \ln \ca^{-\frac{1}{2^p-1}}\right)\, \ln  ( \ca) \, d\mu_K\right) }{\Omega_{-n/2^p} (K)} \nonumber \\
 &=& \lim_{p \to \infty} \frac{2^{2p}}{(2^p-1)^2} \,   \frac{\left(\int_{\partial K} \ca^{-\frac{1}{2^p-1}}\, \ln  ( \ca) \, d\mu_K\right) }{\Omega_{-n/2^p} (K)} \nonumber \\
&=&
\frac{\int_{\partial K}  \ln  ( \ca) \, d\mu_K }{n \, Vol (K)}= \ln (\Lambda_K). \nonumber
\end{eqnarray}
\end{proof}

\

{\bf{Acknowledgements:}}
{\em I am thankful to Monika Ludwig, Vitali Milman and Nicole Tomczak-Jaegermann, the organizers of the 2010 Workshop on Asymptotic Geometric Analysis and Convexity, for the invitation to participate, to the Fields Institute for the hospitality and, to all of the above, for the stimulating atmosphere during my stay there.}

\

\end{document}